\documentclass[11pt,reqno]{amsart}
\usepackage{amsthm, amsfonts, amssymb, color}
 \usepackage{mathrsfs}
\usepackage{enumerate}
\usepackage{amsmath}
 \usepackage{amstext, amsxtra}
  \usepackage{txfonts}
 \usepackage[colorlinks, linkcolor=black, citecolor=blue, hypertexnames=false]{hyperref}
\usepackage{graphicx}
\usepackage[symbol]{footmisc}
\usepackage{cite}
\usepackage{bbm}
 \allowdisplaybreaks
\usepackage{pgfplots}
\pgfplotsset{compat=1.18}

\newtheorem{theorem}{Theorem}[section]
\newtheorem{proposition}[theorem]{Proposition}
\newtheorem{corollary}[theorem]{Corollary}
\newtheorem{lemma}[theorem]{Lemma}

\newtheorem{remark}[theorem]{Remark}

\numberwithin{equation}{section}
\theoremstyle{definition}

\title[]
 {An abstract uncertainty principle with applications}

\author{Tianxiao Huang, Ze Li and Jiani Liu}

\address{Tianxiao Huang, School of Mathematics (Zhuhai), Sun Yat-sen University, Zhuhai, Guangdong 519082, China}
\email{htx5@mail.sysu.edu.cn}

\address{Ze Li, School of Mathematics and Statistics, Ningbo University, Ningbo, Zhejiang 315211, P.R. China }
\email{rikudosennin@163.com}

\address{Jiani Liu, School of Mathematics (Zhuhai), Sun Yat-sen University, Zhuhai, Guangdong 519082, China}
\email{liujn68@mail2.sysu.edu.cn}


\begin{document}

\begin{abstract}
Under Wigdersons' framework and by sorting out the technical points in the recent works of Tang (\textit{J. Fourier Anal. Appl.} \textbf{31} (2025)) and Dias-Luef-Prata (\textit{J. Math. Pures Appl. (9)} \textbf{198} (2025)), we prove an abstract uncertainty principle for functions in the $L^p$ setting. An immediate consequence is a new uncertainty principle for the Fourier transform, unifying and extending many existing results. More applications are shown for PDEs, including the moment growth estimates for some linear and nonlinear dispersive equations, and a type of weighted lower bound estimate for the spacetime moment of the Schr\"{o}dinger equation and heat equation inspired by the control theory.
\end{abstract}

\maketitle


\section{Introduction}\label{sec1}

\subsection{Wigdersons' Framework for the Uncertainty Principle}\

The phrase \textit{uncertainty principle} in mathematics often refers to the fact that \textit{a nonzero function in $\mathbb{R}^n$ and its Fourier transform cannot be too localized to some subset of $\mathbb{R}^n\times\mathbb{R}^n$}. Among these uncertainty principles, the first and most well-known one is the Heisenberg uncertainty principle reflected by the inequality
\begin{equation}\label{e1}
	\left\||x|f\right\|_{L^2(\mathbb{R}^n)}\left\||\xi|\hat{f}\right\|_{L^2(\mathbb{R}^n)}\gtrsim_n\|f\|_{L^2(\mathbb{R}^n)}^2,\quad f\in\mathscr{S}(\mathbb{R}^n),
\end{equation}
where $\hat{f}(\xi)=\int e^{-ix\cdot\xi}f(x)dx$ is the Fourier transform of $f$. Despite of the optimal constant, the classical proofs to \eqref{e1} are closely related to two facts. First, the Fourier transform maps $i^{-1}\partial_jf$ to $\xi_j\hat{f}$, which turns \eqref{e1} into
\begin{equation}\label{e2}
	\left\||x|f\right\|_{L^2(\mathbb{R}^n)}\left\|\nabla f\right\|_{L^2(\mathbb{R}^n)}\gtrsim_n\|f\|_{L^2(\mathbb{R}^n)}^2,\quad f\in\mathscr{S}(\mathbb{R}^n).
\end{equation}
Second and considerably, the canonical commutation relation $[\partial_j,x_j]=1$ helps the proof of \eqref{e2} either in a concrete way using integration by parts, or in the view of Hilbert space techniques.

However in 2021, Avi Wigderson and Yuval Wigderson \cite{WW} found a way around such point of view, and established a new framework of showing how inequalities like \eqref{e1} can hold, even when $\hat{f}$ is replaced by $Af$ where $A$ is a linear operator.

The framework is a two-step mechanism, and we give a short review of it in the setting  of $A$ being a linear operator mapping functions in $\mathbb{R}^n$.

The first step is to prove a \textit{primary uncertainty principle} for $A$, which may appear in the form of a norm inequality
\begin{equation}\label{e3}
	\frac{\|f\|_{L^p(\mathbb{R}^n)}}{\|f\|_{L^q(\mathbb{R}^n)}}\frac{\|Af\|_{L^p(\mathbb{R}^n)}}{\|Af\|_{L^q(\mathbb{R}^n)}}\gtrsim1.
\end{equation}
If we denote
\begin{equation*}
	H^{(0)}(f)=\frac{\|f\|_{L^p(\mathbb{R}^n)}}{\|f\|_{L^q(\mathbb{R}^n)}},
\end{equation*}
and view $H^{(0)}(f)$ as a kind of \textit{measurement of localization} of $f$, then \eqref{e3} reads
\begin{equation}\label{e4}
	H^{(0)}(f)H^{(0)}(Af)\gtrsim1,
\end{equation}
and this is why \eqref{e3} can be called as an uncertainty principle.

The second step is to compare some other measurement of localization $H^{(1)}(f)$ with $H^{(0)}(f)$. If we can prove for all reasonable $f$ and some $\kappa>0$ that
\begin{equation*}
	H^{(1)}(f)\gtrsim\left(H^{(0)}(f)\right)^\kappa,
\end{equation*}
then combining \eqref{e4}, we have an \textit{upgraded uncertainty principle} for $A$
\begin{equation*}
	H^{(1)}(f)H^{(1)}(Af)\gtrsim1.
\end{equation*}
Note that the second step has nothing essential to do with $A$.

Wigdersons in \cite{WW} have already shown when $n=1$ that such framework and its variants can be used to prove the Heisenberg uncertainty principle \eqref{e1}, as well as many "Heisenberg-like" uncertainty principles. More concretely, denoted by
\begin{equation*}
	H^{(0)}_{q}(f)=\frac{\|f\|_{L^1(\mathbb{R})}}{\|f\|_{L^q(\mathbb{R})}},\quad H^{(1)}_{r,q}(f)=\frac{\||x|^rf\|_{L^2(\mathbb{R})}}{\|f\|_{L^q(\mathbb{R})}},
\end{equation*}
and assuming $A$ is a $k$-Hadamard operator which satisfies
\begin{equation}\label{e5}
	\|A\|_{1\rightarrow\infty}\lesssim1\quad\text{and}\quad\|A^*Af\|_{\infty}\geq k\|f\|_{\infty}\text{ for $f\in L^1(\mathbb{R})$ with $Af\in L^1(\mathbb{R})$},
\end{equation}
a primary uncertainty principle for $A$ was first shown:
\begin{equation}\label{e6}
	H^{(0)}_q(f)H^{(0)}_q(Af)\geq k^{1-\frac1q},\quad\text{if }1\leq q\leq\infty,
\end{equation}
and then the following comparison of localization was also proved:
\begin{equation}\label{e7}
	(H^{(1)}_{r,q}(f))^\frac{2q-2}{2qr+q-2}\gtrsim H^{(0)}_q(f),\quad\text{if }1<q\leq\infty,\,\mbox{$\frac12$}<r<\infty.
\end{equation}
These together upgrade the uncertainty principle for $A$ (see \cite[Theorem 4.14]{WW}):
\begin{equation}\label{e8}
	\left(H^{(1)}_{r,q}(f)\right)^\frac{2q-2}{2qr+q-2}\left(H^{(1)}_{s,q}(Af)\right)^\frac{2q-2}{2qs+q-2}\gtrsim k^{1-\frac1q},\quad\text{if }q\in(1,\infty],\,\mbox{$\frac12$}<r,s<\infty.
\end{equation}
In particular, when $A=\mathscr{F}$, $r=s=1$ and $q=2$, \eqref{e8} recovers the Heisenberg uncertainty \eqref{e1} when $n=1$.

Due to a technical problem in proving a higher dimensional comparison result like \eqref{e7}, the authors of \cite{WW} pointed out that their proof does not work when $n>1$, and raised the problem that if one can still prove the higher dimensional Heisenberg uncertainty principle under their framework. This was later resolved independently in Tang \cite{T24,T25} and in Dias-Luef-Prata \cite{DLP}, where the starting point is that, if $A$ is the Fourier transform $\mathscr{F}$, then replacing \eqref{e5} by the Hausdorff-Young inequality $\|\mathscr{F}\|_{p\rightarrow p'}\lesssim1$ for $\mathscr{F}$ and $\|\mathscr{F}^{-1}\|_{p\rightarrow p'}\lesssim1$ ($\frac1p+\frac{1}{p'}=1$) for the inverse will first give many more interesting primary uncertainty principles than \eqref{e6} if the H\"{o}lder's inequality could be properly used, such that the comparison of localization $H_{1,2}^{(1)}(f)$ and $H_2^{(0)}(f)$ is possible in higher dimensions. Such observation yields many new uncertainty principles for the Fourier transform. For example, Tang in \cite{T24,T25} showed for all $n\in\mathbb{N}_+$ and $p\in(1,\frac{2n}{n-1})$ that
\begin{equation}\label{e1.9}
	\frac{\left\||x|f\right\|_{L^p(\mathbb{R}^n)}}{\|f\|_{L^p(\mathbb{R}^n)}}\frac{\left\||\xi|\hat{f}\right\|_{L^p(\mathbb{R}^n)}}{\|\hat{f}\|_{L^p(\mathbb{R}^n)}}\gtrsim1,\quad0\neq f\in L^p\cap\mathscr{F}^{-1}L^p,
\end{equation}
which clearly implies the Heisenberg uncertainty principle \eqref{e1} in any dimension by taking $p=2$.

The same logic applies if a linear operator $A$ satisfies the Hausdorff-Young type inequality $\|A\|_{p\rightarrow p'}\lesssim1$ as well as $\|A^{-1}\|_{p\rightarrow p'}\lesssim1$ for the inverse, and this was the major concern in Dias-Luef-Prata \cite{DLP} where the authors raised a condition called \textit{special $k$- Hadamard operator} (see \cite[Definition 2]{DLP}) to insure these two estimates. A lot of interesting results were obtained in \cite{DLP}, for example various uncertainty principles in the form of
\begin{equation*}
	\left\||x|^\theta f\right\|_{L^p(\mathbb{R}^n)}^\alpha\left\||\xi|^\phi Af\right\|_{L^q(\mathbb{R}^n)}^\beta\gtrsim\|f\|_{L^r(\mathbb{R}^n)}^\alpha\|Af\|_{L^r(\mathbb{R}^n)}^\beta,
\end{equation*}
for various values of $\alpha,\beta,\theta,\phi,p,q,r>0$. We will discuss this and make a comparison with our results later.

\noindent\textbf{Notations}

We denote $A\lesssim B$ if there exists a constant $C>0$ such that $A\leq CB$, and denote $A\lesssim_{\alpha}B$ if $C$ depends on parameter $\alpha$. We denote $A\sim B$ if $A\lesssim B$ and $B\lesssim A$. For $p\in[1,\infty]$, we use $p'$ to denote the conjugate number of $p$ satisfying $\frac1p+\frac{1}{p'}=1$. We use $\mathscr{S}(\mathbb{R}^n)$ to denote the Schwartz class in $\mathbb{R}^n$, and use $\hat{f}(\xi)=\int_{\mathbb{R}^n}e^{-ix\cdot\xi}f(x)dx$ to define the Fourier transform of $f$. We also use $\mathscr{F}$ and $\mathscr{F}^{-1}$ to denote the Fourier transform and its inverse.

Let $\varphi(\xi)$ be a radial bump function supported in the ball $\{\xi\in\mathbb{R}^n;|\xi|\leq\frac{11}{10}\}$ and equal to $1$ on the ball $\{\xi\in\mathbb{R}^n;|\xi|\leq1\}$. For each number $N>0$, we define the Littlewood-Paley projection $P_N$ via the Fourier multiplier
\begin{equation*}
	\widehat{P_Nf}(\xi)=(\varphi(\xi/N)-\varphi(2\xi/N))\hat{f}(\xi).
\end{equation*}

\subsection{Motivation and Main Result}\

The purpose of the present paper is to maximize the utilization of the techniques developed in Tang \cite{T24,T25} and Dias-Luef-Prata \cite{DLP} under Wigdersons' framework, and give a more general uncertainty principle for functions in the $L^p$ setting.

Before introducing the main result, we make a brief regarding the different consideration in this paper. Instead of considering a linear operator $A$ in $\mathbb{R}^n$ who and whose inverse satisfy the Hausdorff-Young type inequality, we will just consider two functions $f_1$ and $f_2$ defined in $\mathbb{R}^{n_1}$ and $\mathbb{R}^{n_2}$ respectively with maybe different dimensions, satisfying
\begin{equation}\label{e2.4}
	\|f_2\|_{L^{m_2}(\mathbb{R}^{n_2})}\leq C_1\|f_1\|_{L^{q_1}(\mathbb{R}^{n_1})},\quad\|f_2\|_{L^{q_2}(\mathbb{R}^{n_2})}\geq C_2\|f_1\|_{L^{m_1}(\mathbb{R}^{n_1})}.
\end{equation}
When $f_2=Af_1$ where $A$ is a linear operator, these two estimates are the boundedness of $A$ and $A^{-1}$. Unlike the previous work \cite{DLP} where the boundedness were assumed to be the Hausdorff-Young type, here $m_1,m_2,q_1,q_2$ are unknown and will only be determined in every specific application. Note that these two estimates already give a primary uncertainty principle
\begin{equation}\label{e1.11}
	\frac{\|f_1\|_{L^{q_1}(\mathbb{R}^{n_1})}}{\|f_1\|_{L^{m_1}(\mathbb{R}^{n_1})}}\frac{\|f_2\|_{L^{q_2}(\mathbb{R}^{n_2})}}{\|f_2\|_{L^{m_2}(\mathbb{R}^{n_2})}}\geq\frac{C_2}{C_1}.
\end{equation}
In particular, the dependence between $f_1$ and $f_2$ is not necessarily linear, so our theory may serve nonlinear problems.

Our main result is the following uncertainty principle.

\begin{theorem}\label{thm1}
	For $i=1,2$, suppose $n_i\in\mathbb{N}_+$, $b_i,q_i\in(0,\infty)$ and $a_i,k_i,m_i\in(0,\infty]$ with
	\begin{equation}\label{e2.3}
		\mbox{$\frac{n_i}{a_i^{-1}n_i+b_i}$}<k_i\leq m_i,\quad\mbox{$\frac{n_i}{a_i^{-1}n_i+b_i}$}<q_i<m_i.
	\end{equation}
	If $0\neq f_i\in L^{q_i}(\mathbb{R}^{n_i})\cap L^{k_i}(\mathbb{R}^{n_i})\cap L^{m_i}(\mathbb{R}^{n_i})$ ($i=1,2$) satisfy \eqref{e2.4} for some $C_1,C_2>0$, then
	\begin{equation}\label{e2.5}
		\begin{split}
			&\left(\frac{\left\||\cdot-x_1|^{b_1}f_1\right\|_{L^{a_1}(\mathbb{R}^{n_1})}}{\|f_1\|_{L^{k_1}(\mathbb{R}^{n_1})}}\right)^{(\frac{1}{q_1}-\frac{1}{m_1})(\frac{1}{a_2}+\frac{b_2}{n_2}-\frac{1}{k_2})}\left(\frac{\left\||\cdot-x_2|^{b_2}f_2\right\|_{L^{a_2}(\mathbb{R}^{n_2})}}{\|f_2\|_{L^{k_2}(\mathbb{R}^{n_2})}}\right)^{(\frac{1}{q_2}-\frac{1}{m_2})(\frac{1}{a_1}+\frac{b_1}{n_1}-\frac{1}{k_1})}\\
			\gtrsim&_{n_i,a_i,b_i,q_i,k_i,m_i\,(i=1,2)}\left(\frac{C_2}{C_1}\right)^{(\frac{1}{a_1}+\frac{b_1}{n_1}-\frac{1}{k_1})(\frac{1}{a_2}+\frac{b_2}{n_2}-\frac{1}{k_2})},\quad x_i\in\mathbb{R}^{n_i},\,i=1,2.
		\end{split}
	\end{equation}
\end{theorem}

Next we go through the applications of this result including Fourier transform and Fourier transform-like operators, but the main focus will be on the the applications to partial differential equations.

\subsection{Uncertainty Principles for the Fourier transform}\

By Hausdorff-Young's inequality, $\|\hat{f}\|_{L^{p'}(\mathbb{R}^n)}\lesssim\|f\|_{L^p(\mathbb{R}^n)}$ and $\|\hat{f}\|_{L^p(\mathbb{R}^n)}\gtrsim\|f\|_{L^{p'}(\mathbb{R}^n)}$ hold for all $p\in[1,2]$. If we apply Theorem \ref{thm1} with $f_1=f$ and $f_2=\hat{f}$, then the following uncertainty principle for the Fourier transform follows as an immediate consequence of Corollary \ref{cor1} and Remark \ref{rk1} in Section \ref{sec2}.

\begin{theorem}\label{thm2}
	Let $n\in\mathbb{N}_+$, $a_i\in(0,\infty]$, $b_i\in(0,\infty)$ with $\frac{n}{a_i^{-1}n+b_i}<2$ ($i=1,2$), and suppose
	\begin{equation*}\small
		k_1\in\begin{cases}
			(\mbox{$\frac{n}{a_1^{-1}n+b_1}$},\infty],\quad&\text{if $\frac{n}{a_2^{-1}n+b_2}<1$},\\
			\left(\mbox{$\frac{n}{a_1^{-1}n+b_1},(\frac{n}{a_2^{-1}n+b_2})'$}\right),&\text{if $\frac{n}{a_2^{-1}n+b_2}\geq1$},
		\end{cases}\quad k_2\in\begin{cases}
			(\mbox{$\frac{n}{a_2^{-1}n+b_2}$},\infty],\quad&\text{if $\frac{n}{a_1^{-1}n+b_1}<1$},\\
			\left(\mbox{$\frac{n}{a_2^{-1}n+b_2},(\frac{n}{a_1^{-1}n+b_1})'$}\right),&\text{if $\frac{n}{a_1^{-1}n+b_1}\geq1$}.
		\end{cases}
	\end{equation*}
	Then for $0\neq f\in L^{k_1}(\mathbb{R}^n)\cap\mathscr{F}^{-1}L^{k_2}(\mathbb{R}^n)$ and $x_0,\xi_0\in\mathbb{R}^n$, we have
	\begin{equation*}
		\left(\frac{\left\||x-x_0|^{b_1}f\right\|_{L_x^{a_1}(\mathbb{R}^n)}}{\|f\|_{L^{k_1}(\mathbb{R}^n)}}\right)^{\frac{1}{a_2}+\frac{b_2}{n}-\frac{1}{k_2}}\left(\frac{\left\||\xi-\xi_0|^{b_1}\hat{f}\right\|_{L_\xi^{a_2}(\mathbb{R}^n)}}{\|\hat{f}\|_{L^{k_2}(\mathbb{R}^n)}}\right)^{\frac{1}{a_1}+\frac{b_1}{n}-\frac{1}{k_1}}\gtrsim_{n,a_i,b_i,k_i}1.
	\end{equation*}
\end{theorem}
\begin{remark}
This theorem as well as its corollaries below also hold when $\hat{f}$ is replaced by $Af$ where $A$ is a linear operator with $\|Af\|_{L^{p'}(\mathbb{R}^n)}\lesssim\|f\|_{L^p(\mathbb{R}^n)}$ and $\|Af\|_{L^p(\mathbb{R}^n)}\gtrsim\|f\|_{L^{p'}(\mathbb{R}^n)}$ for all $p\in[1,2]$, i.e. $A$ and $A^{-1}$ both satisfy the Hausdorff-Young type inequality in the full range.
\end{remark}

For a comparison with those results proved in Dias-Luef-Prata \cite{DLP} and Tang \cite{T24,T25}, we explore some special cases where notations are adapted to those used in \cite{DLP,T24,T25}, and we always assume $x_0=\xi_0=0$.

The first special case is to consider $k_1=k_2$. In Theorem \ref{thm2}, let $a_1=p$, $b_1=\theta$, $a_2=q$, $b_2=\phi$ and $k_1=k_2=r$, we have the follow result generalizing the Cowling-Price uncertainty principle.

\begin{corollary}\label{cor2}
	Let $n\in\mathbb{N}_+$, $\theta,\phi\in(0,\infty)$ and $p,q,r\in(0,\infty]$ with
	\begin{itemize}
		\item[(1)] $\frac{\theta}{n}>\frac{1}{\min\{2,r\}}-\frac1p$, $\frac{\phi}{n}>\frac{1}{\min\{2,r\}}-\frac1q$;
		
		\item[(2)] $r<(\frac{n}{p^{-1}n+\theta})'$ whenever $\frac{n}{p^{-1}n+\theta}\geq1$;
		
		\item[(3)] $r<(\frac{n}{q^{-1}n+\phi})'$ whenever $\frac{n}{q^{-1}n+\phi}\geq1$.
	\end{itemize}
	Then
	\begin{equation}\label{e3.2}\small
		\left(\frac{\left\||x|^\theta f\right\|_{L^p(\mathbb{R}^n)}}{\|f\|_{L^r(\mathbb{R}^n)}}\right)^{\phi-n(\frac1r-\frac1q)}\left(\frac{\left\||\xi|^\phi\hat{f}\right\|_{L^q(\mathbb{R}^n)}}{\|\hat{f}\|_{L^r(\mathbb{R}^n)}}\right)^{\theta-n(\frac1r-\frac1p)}\gtrsim_{n,p,q,r,\theta,\phi}1,\quad0\neq f\in L^r\cap\mathscr{F}^{-1}L^r.
	\end{equation}
\end{corollary}

In particular, when $0<r\leq2$, the conditions (2) and (3) are null for (1) implies $\frac{n}{p^{-1}n+\theta},\frac{n}{q^{-1}n+\phi}<2$. The original Cowling-Price uncertainty principle (see \cite{CP}) is the case $r=2$, $n=1$ where $p,q\in[1,\infty]$, $\theta+\frac1p=\phi+\frac1q$, and the higher dimensional case $r=2$, $n>1$ where $p,q\in(1,\infty)$ were obtained in Tang \cite{T24,T25}, so Corollary \ref{cor2} extends such results to the range $p,q\in(0,1]\cup\{\infty\}$. Corollary \ref{cor2} (with $\hat{f}$ replaced by $Af$) also extends Dias-Luef-Prata \cite[Theorem 18]{DLP} where \eqref{e3.2} was proved when $1\leq r<2$ and $p,q>r$. It is also easy to check that Corollary \ref{cor2} extends \cite[Theorems 14 and 15]{DLP}.

In the next special case, we let $a_1=a_2=p$, $b_1=b_2=\theta$ and $k_1=k_2=q$ in Theorem \ref{thm2} to obtain the following.

\begin{corollary}\label{cor3}
	Let $n\in\mathbb{N_+}$, $\theta\in(0,\infty)$, $p,q\in(0,\infty]$ satisfy $\frac{n}{p^{-1}n+\theta}<\min\{2,q\}$ and
	\begin{equation*}
		\mbox{$q<\left(\frac{n}{p^{-1}n+\theta}\right)'$},\quad\text{whenever }\mbox{$\frac{n}{p^{-1}n+\theta}\geq1$}.
	\end{equation*}
	Then
	\begin{equation}\label{e3.3}
		\frac{\left\||x|^\theta f\right\|_{L^p(\mathbb{R}^n)}}{\|f\|_{L^q(\mathbb{R}^n)}}\frac{\left\||\xi|^\theta\hat{f}\right\|_{L^p(\mathbb{R}^n)}}{\|\hat{f}\|_{L^q(\mathbb{R}^n)}}\gtrsim_{n,p,q,\theta}1,\quad0\neq f\in L^q\cap\mathscr{F}^{-1}L^q.
	\end{equation}
\end{corollary}

\eqref{e3.3} was obtained in Dias-Luef-Prata \cite[Theorem 17]{DLP} when $p\in[1,2)$ and $q\in[p,p']$, and Corollary \ref{cor3} is an extension of such result because clearly $\frac{n}{p^{-1}n+\theta}<p$, so that $p'<(\frac{n}{p^{-1}n+\theta})'$ whenever $\frac{n}{p^{-1}n+\theta}\geq1$.

If $\theta=1$ and $p=2$, then the condition in Corollary \ref{cor3} reads
\begin{equation*}
	q\in\begin{cases}
		(\frac23,\infty],\quad&n=1,\\
		(\frac{2n}{n+2},\frac{2n}{n-2}),&n\geq2,
	\end{cases}
\end{equation*}
which slightly extends Dias-Luef-Prata \cite[Theorem 20]{DLP} where $q\in[1,\infty]$ was assumed when $n=1$, and we remark that the conclusion is not true if $n\geq3$ and $q\in[1,\frac{2n}{n+2})\cup(\frac{2n}{n-2},\infty]$, see \cite[Proposition 22]{DLP}.

The last special case was initially considered in Tang \cite{T25} (or the earlier version \cite{T24}) as mentioned in \eqref{e1.9}. If $a_1=a_2=k_1=k_2=p$ and $b_1=b_2=1$, one checks that Theorem \ref{thm2} or Corollary \ref{cor3} requires $p\in(0,\infty]$, $\frac{n}{p^{-1}n+1}<\min\{2,p\}$, and
\begin{equation*}
	\mbox{$p<\left(\frac{n}{p^{-1}n+1}\right)'$},\quad\text{whenever }\mbox{$\frac{n}{p^{-1}n+1}\geq1$}.
\end{equation*}
But this is equivalent to say $p\in(0,\frac{2n}{n-1})$.

\begin{corollary}
	If $n\in\mathbb{N_+}$ and $p\in(0,\frac{2n}{n-1})$, then
	\begin{equation}\label{e3.4}
		\frac{\left\||x|f\right\|_{L^p(\mathbb{R}^n)}}{\|f\|_{L^p(\mathbb{R}^n)}}\frac{\left\||\xi|\hat{f}\right\|_{L^p(\mathbb{R}^n)}}{\|\hat{f}\|_{L^p(\mathbb{R}^n)}}\gtrsim_{n,p}1,\quad0\neq f\in L^p\cap\mathscr{F}^{-1}L^p.
	\end{equation}
\end{corollary}

When $p=2$, this is the Heisenberg uncertainty principle without optimal constant. Tang \cite[Theorem 2]{T25} obtained \eqref{e3.4} when $p\in(1,\frac{2n}{n-1})$, and it was also proved that \eqref{e3.4} is not true when $p\in(\frac{2n}{n-1},\infty)$. It was asked in \cite{T25} that whether \eqref{e3.4} still holds when $p=1,\frac{2n}{n-1},\infty$, and we partially extend such result to $p\in(0,1]$.

\subsection{Moment Growth for Linear Dispersive Equations}\

Heuristically for dispersive equations, since the dispersion and conservation law force the mass or energy to be ejected to the infinity, the relevant moments should grow in time. We give some examples to show how exactly the moments grow. The idea is that in Theorem \ref{thm1}, if $f_2$ is a solution of a dispersive equation at time $t$ with initial data $f_1$, then \eqref{e2.4} can be chosen essentially from inequalities that give the dispersion and conservation law of the equation.

We first consider the linear Schr\"{o}dinger equation
\begin{equation*}
	\begin{cases}
		i\partial_tu+\Delta u=0,\quad(t,x)\in\mathbb{R}^{1+n},\\
		u(0,x)=u_0.
	\end{cases}
\end{equation*}
Essentially from the dispersive estimate $\|e^{it\Delta}\|_{L^1(\mathbb{R}^n)\rightarrow L^\infty(\mathbb{R}^n)}\lesssim|t|^{-\frac n2}$ and $L^2$ isometry of $e^{it\Delta}$, we can obtain an uncertainty principle for the solution by considering $f_1(x)=u_0(x)$ and $f_2(x)=u(t,x)=e^{it\Delta}u_0(x)$ in Theorem \ref{thm1}.

\begin{theorem}\label{thm3}
	Let $n\in\mathbb{N_+}$, $b_i\in(0,\infty)$, $a_i,k_i\in(0,\infty]$ satisfy $\frac{n}{a_i^{-1}n+b_i}<\min\{2,k_i\}$ ($i=1,2$) and
	\begin{equation*}
		\begin{cases}
			k_1<\left(\frac{n}{a_2^{-1}n+b_2}\right)',\quad\text{whenever }\frac{n}{a_2^{-1}n+b_2}\geq1,\\
			k_2<\left(\frac{n}{a_1^{-1}n+b_1}\right)',\quad\text{whenever }\frac{n}{a_1^{-1}n+b_1}\geq1.
		\end{cases}
	\end{equation*}
	Then for $0\neq u_0\in\mathscr{S}(\mathbb{R}^n)$, we have
	\begin{equation}\label{e4.2}
		\begin{split}
			&\left(\frac{\left\||x-x_0|^{b_1}u_0\right\|_{L_x^{a_1}(\mathbb{R}^n)}}{\|u_0\|_{L^{k_1}(\mathbb{R}^n)}}\right)^{\frac{1}{a_2}+\frac{b_2}{n}-\frac{1}{k_2}}\left(\frac{\left\||x-x_1|^{b_2}e^{it\Delta}u_0\right\|_{L_x^{a_2}(\mathbb{R}^n)}}{\|e^{it\Delta}u_0\|_{L^{k_2}(\mathbb{R}^n)}}\right)^{\frac{1}{a_1}+\frac{b_1}{n}-\frac{1}{k_1}}\\
			\gtrsim&_{n,a_i,b_i,k_i}|t|^{n(\frac{1}{a_1}+\frac{b_1}{n}-\frac{1}{k_1})(\frac{1}{a_2}+\frac{b_2}{n}-\frac{1}{k_2})},\quad t\in\mathbb{R}\setminus\{0\},\,x_0,x_1\in\mathbb{R}^n.
		\end{split}
	\end{equation}
	In particular, taking $k_2=2$, we obtain the moment growth:
	\begin{equation}\label{e4.3}
		\left\||x-x_0|^be^{it\Delta}u_0\right\|_{L_x^a(\mathbb{R}^n)}\gtrsim_{u_0,n,a,b}|t|^{n(\frac1a+\frac bn-\frac12)},\quad t\in\mathbb{R}\setminus\{0\},\,x_0\in\mathbb{R}^n,
	\end{equation}
	provided $a\in(0,\infty]$, $b\in(0,\infty)$ and $\frac{n}{a^{-1}n+b}<2$.
\end{theorem}

\begin{remark}
	The result of course holds if we replace $\Delta$ by a self-adjoint operator $H$ in $L^2(\mathbb{R}^n)$ sharing the same dispersive estimate $\|e^{itH}\|_{L^1(\mathbb{R}^n)\rightarrow L^\infty(\mathbb{R}^n)}\lesssim|t|^{-\frac n2}$. For example, $H=\Delta+V(x)$ where $V$ is a real-valued potential with sufficient decay and smoothness, and $H$ has no zero energy resonance, see e.g. \cite{JSS,S,Y,EG,CHHZ1,CHHZ2} and references therein. If $e^{itH}$ has a different dispersive estimate $\|e^{itH}\|_{L^1(\mathbb{R}^n)\rightarrow L^\infty(\mathbb{R}^n)}\lesssim C(t)$, the same argument in the proof in Section \ref{sec3} can also give the theorem with different bounds in \eqref{e4.2} and \eqref{e4.3}.
\end{remark}

\begin{remark}
	If $b>0$ is an integer, the vector field method can also give
	\begin{equation}\label{bd}
		\left\||x|^be^{it\Delta}u_0\right\|_{L^2(\mathbb{R}^n)}\sim_{n,u_0}|t|^b,\quad|t|\gg1.
	\end{equation}
	see Cazenave \cite[Section 2.5]{C} for $b=1$, and the case of larger integer $b$ can be proved by induction. Our result \eqref{e4.3} coincides with \eqref{bd} in such case.
\end{remark}

The proof of the above theorem for Schr\"{o}dinger equation (see Section \ref{sec3}) can also be adapted to the Airy equation and the Kadomtsev-Petviashnili equations without any difficulty, because the solutions are given by relevant unitary groups in $L^2$, and the dispersive estimates can be found in Koch-Tataru-Vi\c{s}an \cite[Chapter 3 in Part 1]{KTV}.

We next turn to another example of the wave equation
\begin{equation}\label{e4.4}
	\begin{cases}
		\partial_t^2u=\Delta u,\quad(t,x)\in\mathbb{R}^{1+n},\\
		(u(0),\partial_tu(0))=(u_0,u_1)\in\dot{H}\times L^2,
	\end{cases}
\end{equation}
whose conserved energy is
\begin{equation*}
	E(u)(t)=\|\sqrt{-\Delta}u(t)\|_{L_x^2(\mathbb{R}^n)}^2+\|\partial_tu(t)\|^2_{L_x^2(\mathbb{R}^n)}.
\end{equation*}
Since the dispersive estimate for \eqref{e4.4} (see Section \ref{sec3}) depends on the Fourier frequencies, we have to introduce the Littlewood-Paley projection $P_N$ of frequencies $|\xi|\sim N\in2^{\mathbb{Z}}$. We can obtain the uncertainty principle for the projected energy of \eqref{e4.4} by considering $f_1=|\sqrt{-\Delta}P_Nu_0|+|P_Nu_1|$ and $f_2=|\sqrt{-\Delta}P_Nu|+|\partial_tP_Nu|$ in Theorem \ref{thm1}, but note that the dependence between $f_1$ and $f_2$ is not linear.

\begin{theorem}\label{thm4}
	Let $n\in\mathbb{N_+}$, $b_i\in(0,\infty)$, $a_i,k_i\in(0,\infty]$ satisfy $\frac{n}{a_i^{-1}n+b_i}<\min\{2,k_i\}$ ($i=1,2$) and
	\begin{equation*}
		\begin{cases}
			k_1<\left(\frac{n}{a_2^{-1}n+b_2}\right)',\quad\text{whenever }\frac{n}{a_2^{-1}n+b_2}\geq1,\\
			k_2<\left(\frac{n}{a_1^{-1}n+b_1}\right)',\quad\text{whenever }\frac{n}{a_1^{-1}n+b_1}\geq1.
		\end{cases}
	\end{equation*}
	If $(u,\partial_tu)\in C(\mathbb{R}:\dot{H}\times L^2)$ solves \eqref{e4.4}, then
	\begin{equation}\footnotesize\label{e1.22}
		\begin{split}
			&\left(\frac{\left\||x-x_0|^{b_1}\left(|\sqrt{-\Delta}P_Nu_0|+|P_Nu_1|\right)\right\|_{L_x^{a_1}(\mathbb{R}^n)}}{\left\||\sqrt{-\Delta}P_Nu_0|+|P_Nu_1|\right\|_{L^{k_1}(\mathbb{R}^n)}}\right)^{\frac{1}{a_2}+\frac{b_2}{n}-\frac{1}{k_2}}\left(\frac{\left\||x-x_1|^{b_2}\left(|\sqrt{-\Delta}P_Nu|+|\partial_tP_Nu|\right)\right\|_{L_x^{a_2}(\mathbb{R}^n)}}{\left\||\sqrt{-\Delta}P_Nu|+|\partial_tP_Nu|\right\|_{L^{k_2}(\mathbb{R}^n)}}\right)^{\frac{1}{a_1}+\frac{b_1}{n}-\frac{1}{k_1}}\\
			\gtrsim&_{n,a_i,b_i,k_i}\left((1+|t|N)^{n-1}N^{-2n}\right)^{(\frac{1}{a_1}+\frac{b_1}{n}-\frac{1}{k_1})(\frac{1}{a_2}+\frac{b_2}{n}-\frac{1}{k_2})},\quad t\in\mathbb{R}\setminus\{0\},\,N\in2^{\mathbb{Z}},\,x_0,x_1\in\mathbb{R}^n,
		\end{split}
	\end{equation}
	provided $\sqrt{-\Delta}P_Nu_0,P_Nu_1\in L^{k_1}(\mathbb{R}^n)$ and $\sqrt{-\Delta}P_Nu,\partial_tP_Nu\in L^{k_2}(\mathbb{R}^n)$. In particular, taking $k_1=k_2=2$, the conservation of energy implies the moment growth for the projected energy:
	\begin{equation*}
		\begin{split}
			&\left\||x-x_0|^b\left(|\sqrt{-\Delta}P_Nu|+|\partial_tP_Nu|\right)\right\|_{L_x^a(\mathbb{R}^n)}\\
			\gtrsim&_{n,a,b,u_0,u_1}\left((1+|t|N)^{n-1}N^{-2n}\right)^{\frac1a+\frac bn-\frac{1}{2}},\quad t\in\mathbb{R}\setminus\{0\},\,N\in2^{\mathbb{Z}},\,x_0\in\mathbb{R}^n
		\end{split}
	\end{equation*}
	provided $a\in(0,\infty]$, $b\in(0,\infty)$, $\frac{n}{a^{-1}n+b}<2$ and that $|\sqrt{-\Delta}P_Nu_0|+|P_Nu_1|$ has sufficient decay.
\end{theorem}

A similar result can also be established for the Klein-Gordon equation
\begin{equation*}
	\begin{cases}
		\partial_t^2u-\Delta u+u=0,\quad(t,x)\in\mathbb{R}^{1+n},\\
		(u(0),\partial_tu(0))\in H^1\times L^2,
	\end{cases}
\end{equation*}
because we have the energy conservation law
\begin{equation*}
	\|\sqrt{I-\Delta}u\|_{L_x^2(\mathbb{R}^n)}^2+\|\partial_tu\|_{L_x^2(\mathbb{R}^n)}^2=\|\sqrt{I-\Delta}u_0\|_{L_x^2(\mathbb{R}^n)}^2+\|u_1\|_{L_x^2(\mathbb{R}^n)}^2,
\end{equation*}
and the dispersive estimates at different frequencies were proved in Guo-Peng-Wang \cite{GPW}. The argument is completely parallel to those for the wave equation, so we do not present the statement here.

\subsection{Spacetime Moment for the Linear Schr\"{o}dinger and Heat Equations}\

We can also give an application of Theorem \ref{thm1} where $n_1\neq n_2$, by considering the Schr\"{o}dinger equation
\begin{equation}\label{e5.1}
	\begin{cases}
		i\partial_tu+\Delta u=0,\quad t\in[0,T],\,x\in\mathbb{R}^n,\\
		u(0,x)=u_0(x).
	\end{cases}
\end{equation}

The proof of the following theorem (see Section \ref{sec4}) will not only use the dispersive estimate and Strichartz estimate to give an upper bound of $u$, but also use the control theory to give a lower bound of the spacetime norm of $u$. Recall that for Schr\"{o}dinger equation \eqref{e5.1}, the \textit{observability inequality} at time $T$:
\begin{equation}\label{e1.24}
	\|e^{it\Delta}u_0\|_{L_{t,x}^2([0,T]\times\Omega)}\gtrsim_{T,\Omega}\|u_0\|_{L^2(\mathbb{R}^n)},
\end{equation}
is valid if $\Omega\subset\mathbb{R}^n$ is an observable set for \eqref{e5.1} at time $T$. For example, $\Omega$ is the complement of a ball in $\mathbb{R}^n$ (see Wang-Wang-Zhang \cite{WWZ}); $\Omega$ is nonempty and open, satisfying the \textit{Geometric Control Condition} or being periodic (See T\"{a}ufer \cite{T}).

Considering $f_1(x)=u_0(x)$ and $f_2(t,x)=\mathbbm{1}_{[0,T]}(t)\mathbbm{1}_{\Omega}(x)u(t,x)$ in Theorem \ref{thm1} gives the following result.

\begin{theorem}\label{thm5}
	Let $n\in\mathbb{N_+}$, $a_i\in(0,\infty]$, $b_i,k_i\in(0,\infty)$ ($i=1,2$) with
	\begin{equation}\label{e1.25}
		\mbox{$\frac{n}{a_1^{-1}n+b_1}$}<k_1\leq2,\quad\mbox{$\frac{n+1}{a_2^{-1}(n+1)+b_2}$}<\min\{2,k_2\},\quad k_2<\mbox{$\frac{2n+4}{n}$}.
	\end{equation}
	If $u=e^{it\Delta}u_0$ with $0\neq u_0\in L^{k_1}(\mathbb{R}^n)$, $\Omega\subset\mathbb{R}^n$ is an observable set, and $u\in L_{t,x}^{k_2}([0,T]\times\Omega)$, then
	\begin{equation}\label{e1.26}
		\begin{split}
			&\left(\frac{\left\||x-x_0|^{b_1}u_0\right\|_{L_x^{a_1}(\mathbb{R}^n)}}{\|u_0\|_{L^{k_1}(\mathbb{R}^n)}}\right)^{\frac{1}{a_2}+\frac{b_2}{n+1}-\frac{1}{k_2}}\left(\frac{\left\||(t-t_0,x-x_1)|^{b_2}u\right\|_{L_{t,x}^{a_2}([0,T]\times\Omega)}}{\|u\|_{L_{t,x}^{k_2}([0,T]\times\Omega)}}\right)^{\frac{1}{a_1}+\frac{b_1}{n}-\frac{1}{k_1}}\\
			\gtrsim&_{T,\Omega,n,a_i,b_i,k_i}1,\quad x_0,x_1\in\mathbb{R}^n,\,t_0\in\mathbb{R}.
		\end{split}
	\end{equation}
\end{theorem}
	
A parallel result can also be established for the heat equation
\begin{equation*}
	\begin{cases}
		\partial_tu-\Delta u=0,\quad t\in(0,T],\,x\in\mathbb{R}^n,\\
		u(0,x)=u_0(x).
	\end{cases}
\end{equation*}
Recall in the control theory that the \textit{observability inequality} at time $T$
\begin{equation}\label{e1.28}
	\|e^{t\Delta}u_0\|_{L_{t,x}^2([0,T]\times\Omega)}\gtrsim_{T,\Omega}\|e^{T\Delta}u_0\|_{L^2(\mathbb{R}^n)},
\end{equation}
holds if and only if $\Omega\subset\mathbb{R}^n$ is \textit{thick} (see Wang-Wang-Zhang-Zhang \cite{WWZZ}), i.e. $\Omega$ is measurable, and there exists $\gamma,L>0$ with
\begin{equation*}
	\left|\Omega\cap(x+LQ)\right|\geq\gamma L^n,
\end{equation*}
for every $x\in\mathbb{R}^n$, where $Q$ is the open unit cube in $\mathbb{R}^n$ centered at the origin. If we further know that $|\xi|\leq R$ holds in $\mathrm{supp}\,\hat{u}_0$, then it follows from \eqref{e1.28} that
\begin{equation}\label{e127}
		\|e^{t\Delta}u_0\|_{L_{t,x}^2([0,T]\times\Omega)}\gtrsim_{T,\Omega}e^{-TR^2}\|u_0\|_{L^2(\mathbb{R}^n)}.
\end{equation}
Combining this with the use of the heat kernel, we can prove an analogue of Theorem \ref{thm5} for the heat equation.

\begin{theorem}\label{thm5'}
	Let $n\in\mathbb{N_+}$, $a_i\in(0,\infty]$, $b_i,k_i\in(0,\infty)$ ($i=1,2$), and assume \eqref{e1.25}. If $u=e^{t\Delta}u_0$ with $0\neq u_0\in L^{k_1}(\mathbb{R}^n)$ and $|\xi|<R$ holds in $\mathrm{supp}\,\hat{u}_0$ for some $R>0$, and if $u\in L_{t,x}^{k_2}([0,T]\times\Omega)$ for some thick $\Omega\subset\mathbb{R}^n$, then
\begin{equation}\label{e128}
	\begin{split}
		&\left(\frac{\left\||x-x_0|^{b_1}u_0\right\|_{L_x^{a_1}(\mathbb{R}^n)}}{\|u_0\|_{L^{k_1}(\mathbb{R}^n)}}\right)^{\frac{1}{a_2}+\frac{b_2}{n+1}-\frac{1}{k_2}}\left(\frac{\left\||(t-t_0,x-x_1)|^{b_2}u\right\|_{L_{t,x}^{a_2}([0,T]\times\Omega)}}{\|u\|_{L_{t,x}^{k_2}([0,T]\times\Omega)}}\right)^{\frac{1}{a_1}+\frac{b_1}{n}-\frac{1}{k_1}}\\
		\gtrsim&_{T,\Omega,n,a_i,b_i,k_i}e^{-TR^2(\frac{1}{a_1}+\frac{b_1}{n}-\frac{1}{k_1})(\frac{1}{a_2}+\frac{b_2}{n+1}-\frac{1}{k_2})},\quad x_0,x_1\in\mathbb{R}^n,\,t_0\in\mathbb{R}.
	\end{split}
\end{equation}
\end{theorem}

\subsection{Moment Growth for a Nonlinear Schr\"{o}dinger Equation}\

Finally we give an example of applying Theorem \ref{thm1} where $f_2$ depends on $f_1$ in a nonlinear way. Consider a nonlinear Schr\"{o}dinger equation in the form of
\begin{equation}\label{e6.3}
	\begin{cases}
		i\partial_tu+\Delta u+\phi(t,x)|u|^{m-1}u=0,\quad(t,x)\in\mathbb{R}^{1+n},\\
		u(0,x)=u_0.
	\end{cases}
\end{equation}
The point is that, when considering $f_1(x)=u_0$ and $f_2(x)=u(t,x)$ in Theorem \ref{thm1}, if $\phi$ is real-valued, then the $L^2$ mass of the solution is conserved, and sufficiently good condition on $\phi$ will give an estimate of the solution $u$ the same as that of the free Schr\"{o}dinger equation, hence the situation is essentially reduced to the application in Theorem \ref{thm3}.

\begin{proposition}\label{prop6}
	Given $p\in\begin{cases}
		(2,\infty],\,&n=1,\\
		(2,\frac{2n}{n-2}),\,&n\geq2,
	\end{cases}$ $m\in\left(1,\min\{1+\frac4n,p-1\}\right]$, and real-valued $\phi\in L_{t,x}^\infty(\mathbb{R}^{1+n})$ which fulfills
\begin{align*}
\|\phi(t,x)\|_{L^{\frac{p}{p-m-1}}_x(\mathbb{R}^{n})}\leq C |t|^{\frac{nm}{2}(1-\frac2p)}\eta(t),\quad t\in\mathbb{R},
\end{align*}
for some constant $C>0$ and some non-negative valued function $\eta(t)$ satisfying
\begin{equation*}
	\int_{\mathbb{R}}\eta(t)dt\lesssim1\quad\text{and}\quad\sup\limits_{\frac{|t|}{2}\leq|s|\leq|t|,t\in\mathbb{R}}|t|\eta(s)\lesssim1.
\end{equation*}
Consider the Cauchy problem \eqref{e6.3} where $u_0\in L^{p'}(\mathbb{R}^n)\cap L^2(\mathbb{R}^n)$ with $0<\|u_0\|_{L^{p'}(\mathbb{R}^n)}\ll1$, and further assume $\|u_0\|_{L^2(\mathbb{R}^n)}\ll1$ when $m=1+\frac4n$.
	
	If $a_i\in(0,\infty]$ and $b_i,k_i\in(0,\infty)$ ($i=1,2$) satisfy
	\begin{equation*}
		\mbox{$\frac{n}{a_1^{-1}n+b_1}$}<p'\leq k_1\leq2\quad\text{and}\quad\mbox{$\frac{n}{a_2^{-1}n+b_2}$}<2\leq k_2\leq p,
	\end{equation*}
	then the unique solution $u\in C(\mathbb{R}:L^2(\mathbb{R}^n))$ to \eqref{e6.3} exists and satisfies
	\begin{equation*}
		\begin{split}
			&\left(\frac{\left\||x-x_0|^{b_1}u_0\right\|_{L_x^{a_1}(\mathbb{R}^n)}}{\|u_0\|_{L^{k_1}(\mathbb{R}^n)}}\right)^{\frac{1}{a_2}+\frac{b_2}{n}-\frac{1}{k_2}}\left(\frac{\left\||x-x_1|^{b_2}u(t,x)\right\|_{L_x^{a_2}(\mathbb{R}^n)}}{\|u(t,x)\|_{L_x^{k_2}(\mathbb{R}^n)}}\right)^{\frac{1}{a_1}+\frac{b_1}{n}-\frac{1}{k_1}}\\
			\gtrsim&_{n,a_i,b_i,k_i}|t|^{n(\frac{1}{a_1}+\frac{b_1}{n}-\frac{1}{k_1})(\frac{1}{a_2}+\frac{b_2}{n}-\frac{1}{k_2})},\quad t\in\mathbb{R}\setminus\{0\},\,x_0,x_1\in\mathbb{R}^n.
		\end{split}
	\end{equation*}
	In particular, taking $k_2=2$, the $L^2$ mass conservation of \eqref{e6.3} gives
	\begin{equation*}
		\left\||x-x_0|^bu(t,x)\right\|_{L_x^a(\mathbb{R}^n)}\gtrsim_{u_0,n,a,b}|t|^{n(\frac1a+\frac bn-\frac12)},\quad t\in\mathbb{R}\setminus\{0\},\,x_0\in\mathbb{R}^n,
	\end{equation*}
	provided $a\in(0,\infty]$, $b\in(0,\infty)$, $\frac{n}{a^{-1}n+b}<2$ and that $u_0$ has sufficient decay.
\end{proposition}

\subsection{Plan of the paper}\

We will prove Theorem \ref{thm1} in Section \ref{sec2}, Theorems \ref{thm3} and \ref{thm4} in Section \ref{sec3}, Theorems \ref{thm5} and \ref{thm5'} in Section \ref{sec4}, and Proposition \ref{prop6} in Section \ref{sec5}.

\section{Proof of Theorem \ref{thm1}}\label{sec2}

Theorem \ref{thm1} can be proved using two easy lemmas.

The first one has actually been proved in Dias-Luef-Prata \cite[Theorem 9]{DLP} in the non-endpoint case $p<a\neq\infty$. We give a very concise proof here.

\begin{lemma}\label{lm1}
	Let $n\in\mathbb{N}_+$, $a\in(0,\infty]$, $b\in(0,\infty)$, $s\in[1,\infty)$ and $p\in(\frac{n}{a^{-1}n+b},a]\setminus\{\infty\}$. It follows for all $0\neq f\in L^p(\mathbb{R}^n)\cap L^{ps'}(\mathbb{R}^n)$ and $x_0\in\mathbb{R}^n$ that
	\begin{equation}\label{e2.1}
		\left\||x-x_0|^bf\right\|_{L_x^a(\mathbb{R}^n)}\gtrsim_{n,a,b,s,p}\|f\|_{L^p(\mathbb{R}^n)}\left(\frac{\|f\|_{L^p(\mathbb{R}^n)}}{\|f\|_{L^{ps'}(\mathbb{R}^n)}}\right)^{s\left((\frac1a+\frac bn)p-1\right)}.
	\end{equation}
\end{lemma}
\begin{proof}
	We may assume $x_0=0$ by translation. Set
	\begin{equation*}
		T=v_n^{-\frac1n}\left(\frac{\|f\|_{L^p(\mathbb{R}^n)}}{2\|f\|_{L^{ps'}(\mathbb{R}^n)}}\right)^\frac{ps}{n},\quad v_n=\int_{|x|\leq1}dx,
	\end{equation*}
	then
	\begin{equation*}
		\int_{|x|\leq T}|f|^pdx\leq(T^nv_n)^\frac1s\|f\|_{L^{ps'}(\mathbb{R}^n)}^p=\mbox{$\frac12$}\|f\|_{L^p(\mathbb{R}^n)}^p,
	\end{equation*}
	which means
	\begin{equation*}
		\begin{split}
			\mbox{$\frac12$}\|f\|_{L^p(\mathbb{R}^n)}^p\leq\int_{|x|\geq T}|f|^pdx\leq\left\||x|^{-bp}\right\|_{L^\frac{1}{1-a^{-1}p}(|x|\geq T)}\left\||x|^bf\right\|_{L^a(\mathbb{R}^n)}^p.
		\end{split}
	\end{equation*}
	Since $p\in(\frac{n}{a^{-1}n+b},a]\setminus\{\infty\}$ implies $\left\||x|^{-bp}\right\|_{L^\frac{1}{1-a^{-1}p}(|x|\geq T)}\sim T^{n-(a^{-1}n+b)p}$, \eqref{e2.1} immediately follows.
\end{proof}

The second lemma is nothing but H\"{o}lder's inequality.

\begin{lemma}\label{lm2}
	Let $k,p,q,m\in(0,\infty]$ satisfy
	\begin{equation*}
		\begin{cases}
			m\leq k\leq p,\\
			m<q\leq p,
		\end{cases}\quad\text{or}\quad\begin{cases}
		p\leq k\leq m,\\
		p\leq q<m,
	\end{cases}
	\end{equation*}
	then
	\begin{equation}\label{e2.2}
		\frac{\|f\|_{L^p(\mathbb{R}^n)}}{\|f\|_{L^k(\mathbb{R}^n)}}\geq\left(\frac{\|f\|_{L^q(\mathbb{R}^n)}}{\|f\|_{L^m(\mathbb{R}^n)}}\right)^\frac{p^{-1}-k^{-1}}{q^{-1}-m^{-1}},\quad0\neq f\in\mathscr{S}(\mathbb{R}^n).
	\end{equation}
\end{lemma}
\begin{proof}
	Since $m\neq p$ and $m\neq q$, one checks with the H\"{o}lder's inequality that
	\begin{equation*}
		\|f\|_{L^k(\mathbb{R}^n)}\leq\|f\|_{L^p(\mathbb{R}^n)}^\frac{m^{-1}-k^{-1}}{m^{-1}-p^{-1}}\|f\|_{L^m(\mathbb{R}^n)}^\frac{k^{-1}-p^{-1}}{m^{-1}-p^{-1}},
	\end{equation*}
	and that
	\begin{equation*}
		\|f\|_{L^q(\mathbb{R}^n)}^\frac{p^{-1}-k^{-1}}{q^{-1}-m^{-1}}\leq\|f\|_{L^p(\mathbb{R}^n)}^\frac{k^{-1}-p^{-1}}{m^{-1}-p^{-1}}\|f\|_{L^m(\mathbb{R}^n)}^\frac{(k^{-1}-p^{-1})(p^{-1}-q^{-1})}{(m^{-1}-p^{-1})(q^{-1}-m^{-1})}.
	\end{equation*}
	Multiplying them gives \eqref{e2.2}.
\end{proof}

Now we are ready to prove Theorem \ref{thm1}.

\begin{proof}[Proof of Theorem \ref{thm1}]
	\eqref{e2.3} guarantees the existence of $p_i\in(\frac{n_i}{a_i^{-1}n_i+b_i},\min\{a_i,k_i,q_i\}]$ and $s_i\in[1,\infty)$ with $p_is_i'=k_i$. Note that $p_i\leq q_i<\infty$, so we can use Lemma \ref{lm1} and Lemma \ref{lm2} to deduce
	\begin{equation*}
		\begin{split}
			\frac{\left\||\cdot-x_i|^{b_i}f_i\right\|_{L^{a_i}(\mathbb{R}^{n_i})}}{\|f_i\|_{L^{k_i}(\mathbb{R}^{n_i})}}\gtrsim\left(\frac{\|f_i\|_{L^{p_i}(\mathbb{R}^{n_i})}}{\|f_i\|_{L^{k_i}(\mathbb{R}^{n_i})}}\right)^{s_i\left((\frac{1}{a_i}+\frac{b_i}{n_i})p_i-1\right)+1}\gtrsim\left(\frac{\|f_i\|_{L^{q_i}(\mathbb{R}^{n_i})}}{\|f_i\|_{L^{m_i}(\mathbb{R}^{n_i})}}\right)^{\frac{p_i^{-1}-k_i^{-1}}{q_i^{-1}-m_i^{-1}}\left(s_i\left((\frac{1}{a_i}+\frac{b_i}{n_i})p_i-1\right)+1\right)}.
		\end{split}
	\end{equation*}
	One can use $k_i^{-1}=(p_is_i')^{-1}$ to check
	\begin{equation*}
		\frac{p_i^{-1}-k_i^{-1}}{q_i^{-1}-m_i^{-1}}\left(s_i\left((\mbox{$\frac{1}{a_i}+\frac{b_i}{n_i}$})p_i-1\right)+1\right)=\frac{\frac{1}{a_i}+\frac{b_i}{n_i}-\frac{1}{k_i}}{\frac{1}{q_i}-\frac{1}{m_i}},
	\end{equation*}
	and therefore
	\begin{equation*}
		\begin{split}
			&\left(\frac{\left\||\cdot-x_1|^{b_1}f_1\right\|_{L^{a_1}(\mathbb{R}^{n_1})}}{\|f_1\|_{L^{k_1}(\mathbb{R}^{n_1})}}\right)^{(\frac{1}{q_1}-\frac{1}{m_1})(\frac{1}{a_2}+\frac{b_2}{n_2}-\frac{1}{k_2})}\left(\frac{\left\||\cdot-x_2|^{b_2}f_2\right\|_{L^{a_2}(\mathbb{R}^{n_2})}}{\|f_2\|_{L^{k_2}(\mathbb{R}^{n_2})}}\right)^{(\frac{1}{q_2}-\frac{1}{m_2})(\frac{1}{a_1}+\frac{b_1}{n_1}-\frac{1}{k_1})}\\
			\gtrsim&_{n_i,a_i,b_i,q_i,k_i,m_i\,(i=1,2)}\left(\frac{\|f_1\|_{L^{q_1}(\mathbb{R}^{n_1})}}{\|f_1\|_{L^{m_1}(\mathbb{R}^{n_1})}}\frac{\|f_2\|_{L^{q_2}(\mathbb{R}^{n_2})}}{\|f_2\|_{L^{m_2}(\mathbb{R}^{n_2})}}\right)^{(\frac{1}{a_1}+\frac{b_1}{n_1}-\frac{1}{k_1})(\frac{1}{a_2}+\frac{b_2}{n_2}-\frac{1}{k_2})}.
		\end{split}
	\end{equation*}
	\eqref{e2.5} follows by this and \eqref{e1.11}.
\end{proof}

We give a corollary that will be convenient for the proofs in later sections considering $n_1=n_2=n$.
\begin{corollary}\label{cor1}
	Let $n\in\mathbb{N_+}$ $a_i\in(0,\infty]$ and $b_i\in(0,\infty)$ with $\frac{n}{a_i^{-1}n+b_i}<2$ ($i=1,2$). Suppose for some $p_i\in[1,2)\cap(\frac{n}{a_i^{-1}n+b_i},2)$ and $A,B>0$ that
	\begin{equation}\label{e2.6}
		\|f_2\|_{L^{p_1'}(\mathbb{R}^n)}\lesssim A^{\frac{2}{p_1}-1}\|f_1\|_{L^{p_1}(\mathbb{R}^n)},\quad\|f_2\|_{L^{p_2}(\mathbb{R}^n)}\gtrsim B^{\frac{2}{p_2'}-1}\|f_1\|_{L^{p_2'}(\mathbb{R}^n)},
	\end{equation}
	then when $k_1\in(\frac{n}{a_1^{-1}n+b_1},p_2']$ and $k_2\in(\frac{n}{a_2^{-1}n+b_2},p_1']$, we have
	\begin{equation}\label{e2.7}
		\begin{split}
			&\left(\frac{\left\||x-x_1|^{b_1}f_1\right\|_{L_x^{a_1}(\mathbb{R}^{n})}}{\|f_1\|_{L^{k_1}(\mathbb{R}^{n})}}\right)^{\frac{1}{a_2}+\frac{b_2}{n}-\frac{1}{k_2}}\left(\frac{\left\||x-x_2|^{b_2}f_2\right\|_{L_x^{a_2}(\mathbb{R}^{n})}}{\|f_2\|_{L^{k_2}(\mathbb{R}^{n})}}\right)^{\frac{1}{a_1}+\frac{b_1}{n}-\frac{1}{k_1}}\\
			\gtrsim&_{n,a_i,b_i,p_i,k_i}\left(\frac{B^{\frac{2}{p_2'}-1}}{A^{\frac{2}{p_1}-1}}\right)^{\frac{1}{\frac{1}{p_1}-\frac{1}{p_2'}}(\frac{1}{a_1}+\frac{b_1}{n}-\frac{1}{k_1})(\frac{1}{a_2}+\frac{b_2}{n}-\frac{1}{k_2})},\quad x_1,x_2\in\mathbb{R}^n.
		\end{split}
	\end{equation}
\end{corollary}
\begin{proof}
	The statement follows from Theorem \ref{thm1} and the fact that $\frac{1}{p_1}-\frac{1}{p_2'}=\frac{1}{p_2}-\frac{1}{p_1'}$, if we check \eqref{e2.3} by
	\begin{equation*}
		\begin{split}
			\mbox{$\frac{n}{a_1^{-1}n+b_1}<k_1\leq p_2',\quad\frac{n}{a_1^{-1}n+b_1}<p_1<2<p_2',$}\\
			\mbox{$\frac{n}{a_2^{-1}n+b_2}<k_2\leq p_1',\quad\frac{n}{a_2^{-1}n+b_2}<p_2<2<p_1'$}.
		\end{split}
	\end{equation*}
\end{proof}

\begin{remark}\label{rk1}
	If for every $p_i\in[1,2)\cap(\frac{n}{a_i^{-1}n+b_i},2)$ ($i=1,2$), there exist $A,B>0$ for \eqref{e2.6} to hold, then when
	\begin{equation*}
		k_1\in\begin{cases}
			(\mbox{$\frac{n}{a_1^{-1}n+b_1}$},\infty],\quad&\text{if $\frac{n}{a_2^{-1}n+b_2}<1$},\\
			\left(\mbox{$\frac{n}{a_1^{-1}n+b_1},(\frac{n}{a_2^{-1}n+b_2})'$}\right),&\text{if $\frac{n}{a_2^{-1}n+b_2}\geq1$},
		\end{cases}\quad k_2\in\begin{cases}
		(\mbox{$\frac{n}{a_2^{-1}n+b_2}$},\infty],\quad&\text{if $\frac{n}{a_1^{-1}n+b_1}<1$},\\
		\left(\mbox{$\frac{n}{a_2^{-1}n+b_2},(\frac{n}{a_1^{-1}n+b_1})'$}\right),&\text{if $\frac{n}{a_1^{-1}n+b_1}\geq1$},
		\end{cases}
	\end{equation*}
	\eqref{e2.7} is true for those $p_i$ with $k_1\leq p_2'$ and $k_2\leq p_1'$, because such $p_i$ exist in $[1,2)\cap(\frac{n}{a_i^{-1}n+b_i},2)$.
\end{remark}

\section{Proofs of Theorems \ref{thm3} and \ref{thm4}}\label{sec3}

\begin{proof}[Proof of Theorem \ref{thm3}]
Interpolating the dispersive estimate $\|e^{it\Delta}\|_{L^1(\mathbb{R}^n)\rightarrow L^\infty(\mathbb{R}^n)}\lesssim|t|^{-\frac n2}$ and $L^2$ isometry of $e^{it\Delta}$, we have when $p\in[1,2]$ that
\begin{equation*}
	\|e^{it\Delta}u_0\|_{L^{p'}(\mathbb{R}^n)}\lesssim(|t|^{-\frac n2})^{\frac2p-1}\|u_0\|_{L^p(\mathbb{R}^n)},\quad\|e^{it\Delta}u_0\|_{L^p(\mathbb{R}^n)}\gtrsim(|t|^{-\frac n2})^{\frac{2}{p'}-1}\|u_0\|_{L^{p'}(\mathbb{R}^n)}.
\end{equation*}
Set $f_1=u_0$, $f_2=e^{it\Delta}u_0$ and consequently $A=B=|t|^{-\frac n2}$ in Corollary \ref{cor1}. Since
\begin{equation*}
	\left(\frac{B^{\frac{2}{p_2'}-1}}{A^{\frac{2}{p_1}-1}}\right)^\frac{1}{\frac{1}{p_1}-\frac{1}{p_2'}}=|t|^n,
\end{equation*}
we then obtain \eqref{e4.2}.
\end{proof}

\begin{proof}[Proof of Theorem \ref{thm4}]
The solution of \eqref{e4.4} is given by
\begin{equation}\label{e4.5}
	\begin{cases}
		u=\cos(t\sqrt{-\Delta})u_0+\frac{\sin(t\sqrt{-\Delta})}{\sqrt{-\Delta}}u_1,\\
		\partial_tu=-\sin(t\sqrt{-\Delta})\sqrt{-\Delta}u_0+\cos(t\sqrt{-\Delta})u_1.
	\end{cases}
\end{equation}
Since the frequency-localized dispersive estimate for the half-wave propagator (see Koch-Tataru-Vi\c{s}an \cite[Chapter 2 in Part 3]{KTV}) is known when $p\in[1,2]$:
\begin{equation*}
	\left\|e^{\pm it\sqrt{-\Delta}}P_Nf\right\|_{L^{p'}(\mathbb{R}^n)}\lesssim(1+|t|N)^{-\frac{n-1}{2}(\frac2p-1)}N^{n(\frac2p-1)}\|P_Nf\|_{L^p(\mathbb{R}^n)},
\end{equation*}
we have for $p\in[1,2]$ that
\begin{equation*}\small
	\left\||\sqrt{-\Delta}P_Nu|+|\partial_tP_Nu|\right\|_{L^{p'}(\mathbb{R}^n)}\lesssim\left((1+|t|N)^{-\frac{n-1}{2}}N^n\right)^{\frac2p-1}\left\||\sqrt{-\Delta}P_Nu_0|+|P_Nu_1|\right\|_{L^p(\mathbb{R}^n)},
\end{equation*}
and the inverse of \eqref{e4.5} also gives the other direction
\begin{equation*}\small
	\left\||\sqrt{-\Delta}P_Nu|+|\partial_tP_Nu|\right\|_{L^{p}(\mathbb{R}^n)}\gtrsim\left((1+|t|N)^{-\frac{n-1}{2}}N^n\right)^{\frac{2}{p'}-1}\left\||\sqrt{-\Delta}P_Nu_0|+|P_Nu_1|\right\|_{L^{p'}(\mathbb{R}^n)}.
\end{equation*}
Now we can apply Corollary \ref{cor1} and Remark \ref{rk1} with $f_1=|\sqrt{-\Delta}P_Nu_0|+|P_Nu_1|$ and $f_2=|\sqrt{-\Delta}P_Nu|+|\partial_tP_Nu|$ to obtain \eqref{e1.22}.
\end{proof}

\section{Proof of Theorems \ref{thm5} and \ref{thm5'}}\label{sec4}

	\begin{proof}[Proof of Theorem \ref{thm5}]
		
		It is easy to check from the dispersive estimate $\|e^{it\Delta}u_0\|_{L^{p'}(\mathbb{R}^n)}\lesssim t^{-\frac n2(1-\frac{2}{p'})}\|u_0\|_{L^p(\mathbb{R}^n)}$ ($p\in[1,2]$) that
		\begin{equation*}
			\|e^{it\Delta}u_0\|_{L^{p'}([0,T]\times\mathbb{R}^n)}\lesssim_{p,T}\|u_0\|_{L^p(\mathbb{R}^n)},\quad\text{if }\mbox{$\frac{2n+2}{n+2}<p\leq2$}.
		\end{equation*}
		Interpolating this with the Strichartz estimate $\|e^{it\Delta}u_0\|_{L^\frac{2n+4}{n}([0,T]\times\mathbb{R}^n)}\lesssim\|u_0\|_{L^2(\mathbb{R}^n)}$, we further have
		\begin{equation*}
			\|e^{it\Delta}u_0\|_{L^q([0,T]\times\mathbb{R}^n)}\lesssim_{T,p,q}\|u_0\|_{L^p(\mathbb{R}^n)},\quad(\mbox{$\frac1p,\frac1q$})\in\Delta_n,
		\end{equation*}
		where $\Delta_n$ is the closed triangle with vertices $(\frac12,\frac{n}{2n+4})$, $(\frac12,\frac12)$, $(\frac{n+2}{2n+2},\frac{n}{2n+2})$ minus the point $(\frac{n+2}{2n+2},\frac{n}{2n+2})$ and the open edge connecting $(\frac12,\frac{n}{2n+4})$ and $(\frac{n+2}{2n+2},\frac{n}{2n+2})$. (See Figure 1 below.)
		
		Combing this with \eqref{e1.24}, we now have
		\begin{equation}\label{e5.2}
			\begin{cases}
				\left\|\mathbbm{1}_{[0,T]}(t)\mathbbm{1}_\Omega(x)e^{it\Delta}u_0\right\|_{L^{q}(\mathbb{R}^{1+n})}\lesssim_{T,p,q}\|u_0\|_{L^p(\mathbb{R}^n)},\quad\text{if }(\mbox{$\frac1p,\frac1q$})\in\Delta_n,\\
				\left\|\mathbbm{1}_{[0,T]}(t)\mathbbm{1}_\Omega(x)e^{it\Delta}u_0\right\|_{L^{2}(\mathbb{R}^{1+n})}\gtrsim_{T,\Omega}\|u_0\|_{L^2(\mathbb{R}^n)},
			\end{cases}
		\end{equation}
		so we can apply Theorem \ref{thm1} with $n_1=n$, $n_2=n+1$, $f_1(x)=u_0(x)$ and $f_2(t,x)=\mathbbm{1}_{[0,T]}(t)\mathbbm{1}_\Omega(x)e^{it\Delta}u_0(x)$. One checks that \eqref{e1.25} guarantees the existence of a pair $(p,q)$ with $(\mbox{$\frac1p,\frac1q$})\in\Delta_n$, such that \eqref{e2.3} holds if $m_1=2$, $m_2=q$, $q_1=p$ and $q_2=2$, i.e.
		\begin{equation}\label{eee4.2}
			\begin{split}
				\mbox{$\frac{n}{a_1^{-1}n+b_1}<k_1\leq2$},& \quad\mbox{$\frac{n}{a_1^{-1}n+b_1}<p<2$},\\
				\mbox{$\frac{n+1}{a_2^{-1}(n+1)+b_2}<k_2\leq q$},&\quad\mbox{$\frac{n+1}{a_2^{-1}(n+1)+b_2}<2<q$}.
			\end{split}
		\end{equation}
		So Theorem \ref{thm1} with \eqref{e5.2} gives the statement.
	\end{proof}
	
	\begin{center}
	\begin{tikzpicture}
		[
		labelstyle/.style={font=\footnotesize},
		captionstyle/.style={font=\scriptsize, anchor=north}] 
		\begin{axis}[
			xlabel=$\frac{1}{p}$,
			ylabel=$\frac{1}{q}$,
			xlabel style={at={(1.05,0)}, anchor=north east},
			ylabel style={at={(0,0.95)}, anchor=south east, rotate=270},
			xmin=0, xmax=1.1,
			ymin=0, ymax=1.1,
			xtick={0.5, 0.625, 1},
			xticklabels={$\frac{1}{2}$, $\frac{n+2}{2n+2}$, 1},  
			ytick={0.3, 0.375, 0.5, 1},
			yticklabels={$\frac{n}{2n+4}$, $\frac{n}{2n+2}$, $\frac{1}{2}$, 1},  
			axis lines=left,
			grid=none,
			width=8cm,
			height=8cm,
			]
			
			\fill[gray!30] (0.5,0.3) -- (0.5,0.5) -- (0.625,0.375) -- cycle;
			
			\draw[solid] (0.5,0.3) -- (0.5,0.5);
			\draw[solid] (0.5,0.5) -- (0.625,0.375);
			\draw[dashed] (0.5,0.3) -- (0.625,0.375);
			
			\draw[dashed] (0.5,0.5) -- (0,0.5);
			\draw[dashed] (0.5,0.3) -- (0.5,0);
			\draw[dashed] (0.5,0.3) -- (0,0.3);
			\draw[dashed] (0,0.375) -- (0.625,0.375);
			\draw[dashed] (0.625,0) -- (0.625,0.375);
			\draw[dashed] (0,1) -- (0.5,0.5);
			\draw[dashed] (0.625,0.375) -- (1,0);
			
			\node[circle, fill, inner sep=0.8pt, label={[labelstyle]below left:$B$}] at (0.5,0.3) {};
			\node[circle, fill, inner sep=0.8pt, label={[labelstyle]above:$A$}] at (0.5,0.5) {};
			\node[circle, draw=black, fill=white, inner sep=0.8pt, label={[labelstyle]right:$C$}] at (0.625,0.375) {};
		\end{axis}
		
		\node[label={[labelstyle]below left:$0$}] at (0.1,0.1) {}; 
		
		\node[captionstyle] at (current bounding box.south) 
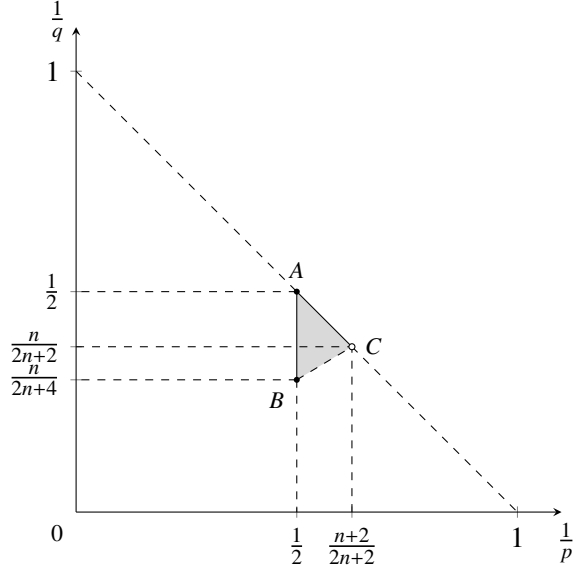
		{Figure1: $\Delta_n$ is the closed triangle $ABC$ without point $C$ and without the open edge $BC$.};
		
	\end{tikzpicture}
	\end{center}

	\begin{proof}[Proof of Theorem \ref{thm5'}]
		The heat kernel $e^{t\Delta}(x-y)\sim t^{-\frac n2}e^{-\frac{|x-y|^2}{4t}}$ first gives
		\begin{equation*}
			\|e^{t\Delta}u_0\|_{L^q(\mathbb{R}^n)}\lesssim t^{-\frac n2(\frac1p-\frac1q)}\|u_0\|_{L^p(\mathbb{R}^n)},\quad\text{if }1\leq p\leq q\leq\infty,
		\end{equation*}
		so it is easy to check that
		\begin{equation}\label{ee4.2}
			\|e^{t\Delta}u_0\|_{L_{t,x}^q([0,T]\times\mathbb{R}^n)}\lesssim_T\|u_0\|_{L^p(\mathbb{R}^n)},\quad\text{if }1\leq p\leq q<\mbox{$\frac{n+2}{n}$}p~\text{or}~p=q=\infty.
		\end{equation}
		Combining \eqref{e127}, now \eqref{e1.25} guarantees the existence of a pair $(p,q)$ in the range described in \eqref{ee4.2}, such that \eqref{e2.3} holds if $m_1=2$, $m_2=q$, $q_1=p$ and $q_2=2$, i.e. \eqref{eee4.2} holds. The statement follows at once if we apply Theorem \ref{thm1} with $n_1=n$, $n_2=n+1$, $f_1(x)=u_0(x)$ and $f_2(t,x)=\mathbbm{1}_{[0,T]}(t)\mathbbm{1}_\Omega(x)e^{t\Delta}u_0(x)$.
	\end{proof}

\section{Proof of Proposition \ref{prop6}}\label{sec5}

The $L^2$ theory first tells us the following (see Linares-Ponce [Theorem 5.2 and Corollary 5.2]\cite{LG} with no change of the proofs).
\begin{lemma}\label{lm3}
	If $m\in(1,1+\frac4n]$, $\phi\in L^\infty(\mathbb{R}^{1+n})$ is real-valued, $u_0\in L^2(\mathbb{R}^n)$, and in addition $\|u_0\|_{L^2}\ll1$ when $m=1+\frac4n$, then the unique solution $u\in C(\mathbb{R}:L^2(\mathbb{R}^n))$ to \eqref{e6.3} exists.
\end{lemma}

We further establish a simple result for the solution in $L^p$. For $p\in(2,\infty]$, define $X_p=\{f(t,x);\|f\|_{X_p}<\infty\}$ where
\begin{equation*}
	\|f\|_{X_p}=\sup_{t\in\mathbb{R}}|t|^{\frac n2(1-\frac2p)}\|f(t,\cdot)\|_{L^p(\mathbb{R}^n)}.
\end{equation*}

\begin{lemma}\label{lm4}
	Assume that $p\in\begin{cases}
	(2,\infty],\,&n=1,\\
	(2,\frac{2n}{n-2}),\,&n\geq2,
	\end{cases}$ $m\in(1,p-1]$, $r=\frac{p}{p-m-1}$ (i.e. $\frac1r+\frac mp=\frac{1}{p'}$), and $\phi(t,x)$ is a measurable function fulfilling
\begin{align*}
\|\phi(t,x)\|_{L_x^r(\mathbb{R}^{n})}\le C |t|^{\frac{nm}{2}(1-\frac2p)}\eta(t),\quad t\in\mathbb{R},
\end{align*}
for some constant $C>0$ and some non-negative valued function $\eta(t)$ satisfying
\begin{align*}
\int_{\mathbb{R}}\eta(t)dt\lesssim 1\quad\text{and}\quad\sup\limits_{\frac{{|t|}}{2} \le |s| \le |t|,t\in\mathbb{R}}|t|\eta(s)\lesssim 1.
\end{align*}
Then there exists $\epsilon=\epsilon(n,\phi,p,m)>0$ such that when $\|u_0\|_{L^{p'}(\mathbb{R}^n)}\leq\epsilon$, \eqref{e6.3} has a unique solution $u$ in $X_p$, and
	\begin{equation*}
		\|u\|_{X_p}\lesssim_{n,\phi,p,m}\|u_0\|_{L^{p'}(\mathbb{R}^n)}.
	\end{equation*}
\end{lemma}
\begin{proof}
	Define operator $T$ by
	\begin{equation*}
		(Tf)(t,x)=e^{it\Delta}u_0+i\int_0^te^{i(t-s)\Delta}\phi(s,x)|f(s,x)|^{m-1}f(s,x)ds.
	\end{equation*}
	We first show that if $\|u_0\|_{L^{p'}(\mathbb{R}^n)}\leq\epsilon$ and $\epsilon$ is small, then $T$ maps the closed ball with radius $2K\epsilon$ in $X_p$ into itself for some constant $K=K(n,\phi,p)>0$. Using $\|e^{it\Delta}\|_{L^{p'}(\mathbb{R}^n)\rightarrow L^p(\mathbb{R}^n)}\lesssim|t|^{-\frac n2(1-\frac2p)}$, we deduce
	\begin{equation*}
		\begin{split}
			\|Tf(t,\cdot)\|_{L^p(\mathbb{R}^n)}&\lesssim|t|^{-\frac n2(1-\frac2p)}\|u_0\|_{L^{p'}(\mathbb{R}^n)}+\left|\int_0^{t}|t-s|^{-\frac n2(1-\frac2p)}\left\|\phi(s,\cdot)|f(s,\cdot)|^m\right\|_{L^p(\mathbb{R}^n)}ds\right|\\
			&\lesssim|t|^{-\frac n2(1-\frac2p)}\|u_0\|_{L^{p'}(\mathbb{R}^n)}+\left|\int_0^{t}|t-s|^{-\frac n2(1-\frac2p)}\|\phi(s,\cdot)\|_{L^r(\mathbb{R}^n)}\|f(s,\cdot)\|_{L^p(\mathbb{R}^n)}^mds\right|\\
			&\lesssim_\phi|t|^{-\frac n2(1-\frac2p)}\|u_0\|_{L^{p'}(\mathbb{R}^n)}+\|f\|_{X_p}^m\left|\int_0^{t}|t-s|^{-\frac n2(1-\frac2p)}\eta(s)ds\right|,
		\end{split}
	\end{equation*}
	and thus
	\begin{equation*}
		\begin{split}
			\|Tf\|_{X_p}&\lesssim_\phi\|u_0\|_{L^{p'}(\mathbb{R}^n)}+\|f\|^m_{X_p}\left|\int_0^{t}|t|^{\frac n2(1-\frac2p)}|t-s|^{-\frac n2(1-\frac2p)}\eta(s)ds\right|\\
			&\lesssim_{\phi,n,p}\|u_0\|_{L^{p'}(\mathbb{R}^n)}+\|f\|^m_{X_p},
		\end{split}
	\end{equation*}
	where we have used
	\begin{align*}
		\left|\int_0^{\frac{t}{2}}|t|^{\frac n2(1-\frac2p)}|t-s|^{-\frac n2(1-\frac2p)}\eta(s)ds\right|\lesssim_{n,p}\left|\int_0^{\frac{t}{2}} \eta(s)ds\right|\lesssim_{n,p}1,\quad \forall t\in\mathbb{R},
	\end{align*}
	and
	\begin{equation*}
		\begin{split}
			\left|\int^t_{\frac{t}{2}}|t|^{\frac n2(1-\frac2p)}|t-s|^{-\frac n2(1-\frac2p)}\eta(s)ds\right|&\lesssim|t|^{\frac n2(1-\frac2p)}\sup_{\frac{|t|}{2}\leq|s|\leq|t|}\eta(s)\left|\int^t_{\frac{t}{2}}|t-s|^{-\frac n2(1-\frac2p)} ds\right|\\
			&\lesssim_{n,p}|t| \sup_{\frac{|t|}{2}\leq|s|\leq|t|}\eta(s)\\
			&\lesssim_{n,p}1,\quad\forall t\in\mathbb{R}.
		\end{split}
	\end{equation*}
	Now we have for some $K=K(n,\phi,p)>0$ that
	\begin{equation*}
		\begin{split}
			\|Tf\|_{X_p}&\leq K\left(\|u_0\|_{L^{p'}(\mathbb{R}^n)}+\|f\|_{X_p}^{m}\right)\\
			&\leq K(\epsilon+(2K\epsilon)^{m})\\
			&\leq2K\epsilon,
		\end{split}
	\end{equation*}
	provided that $\epsilon\ll_{K,m}1$. It is parallel to show that $T$ contracts, so solution $u$ exists in $X_p$ by the fixed point argument, satisfying
	\begin{equation*}
		\|u\|_{X_p}\lesssim K\|u_0\|_{L^{p'}(\mathbb{R}^n)}.
	\end{equation*}
\end{proof}

Combining Lemma \ref{lm3} and Lemma \ref{lm4}, we can now apply Theorem \ref{thm1} with $f_1(x)=u_0(x)$ and $f_2(x)=u(t,x)$ to prove Proposition \ref{prop6}.

\begin{proof}[Proof of Proposition \ref{prop6}]
	Note that since $\phi$ is real-valued, the $L^2$ mass of the solution is conserved. Now Lemma \ref{lm3} and Lemma \ref{lm4} are applicable to identify solution $u$ with property 
	\begin{equation}\label{e6.2}
		\begin{cases}
			\|u(t,\cdot)\|_{L^{p}(\mathbb{R}^n)}\lesssim|t|^{-\frac n2(1-\frac2p)}\|u_0\|_{L^{p'}(\mathbb{R}^n)},\\
			\|u(t,\cdot)\|_{L^2(\mathbb{R}^n)}=\|u_0\|_{L^2(\mathbb{R}^n)},
		\end{cases}
	\end{equation}
	and we are only left to show that Theorem \ref{thm1} is applicable, if we check \eqref{e2.3} by
	\begin{equation*}
		\begin{split}
			\mbox{$\frac{n}{a_1^{-1}n+b_1}$}<k_1\leq2,\quad\mbox{$\frac{n}{a_1^{-1}n+b_1}$}<p'<2,\\
			\mbox{$\frac{n}{a_2^{-1}n+b_2}$}<k_2\leq p,\quad\mbox{$\frac{n}{a_2^{-1}n+b_2}$}<2<p.
		\end{split}
	\end{equation*}
\end{proof}

\begin{remark}
	When $\phi(t,x)\equiv1$, it is possible to still deduce the first dispersive type estimate in \eqref{e6.2} (with constant depending on $\|u_0\|_{L^2(\mathbb{R}^n)}$) in some spatial dimensions and for some range of $p$, even without the small data assumption in the mass-critical case $m=1+\frac4n$, for example see Fan-Killip-Visan-Zhao \cite{FKVZ}.
\end{remark}

\noindent
\section*{Acknowledgements}
This work is supported by the National Key R\&D Program of China under the grant 2023YFA1010300. T. Huang is partially supported by the National Natural Science Foundation of China under the grants 12101621 and 12371244. Z. Li is partially supported by the National Natural Science Foundation of China under the grants 12422110 and 12371244.

T. Huang Thanks Chenjie Fan, Baoping Liu, Yakun Xi and Cheng Zhang for some helpful discussion on the topic. He also thanks Prof. Nicolas Burq for pointing out a mistake in Theorem \ref{thm5'} for an earlier version of the manuscript.

\end{document}